\documentclass[12pt]{article}

\usepackage{amsmath}
\usepackage{amsfonts}
\usepackage{amsthm}
\usepackage{amssymb}
\usepackage{enumerate}
\usepackage{graphicx}
\usepackage{extarrows}
\usepackage{mathtools}
\usepackage[all]{xy}

\usepackage{tikz-cd}
\usepackage{nicefrac}
\usetikzlibrary{matrix, arrows}

\DeclareMathOperator\GL{GL}
\DeclareMathOperator\PGL{PGL}
\DeclareMathOperator\SL{SL}

\DeclareMathOperator\Hom{Hom}

\DeclareMathOperator\id{id}

\DeclareMathOperator\Ext{Ext}
\DeclareMathOperator\ad{ad}
\DeclareMathOperator\coad{coad}
\DeclareMathOperator\cad{cad}
\DeclareMathOperator\cop{cop}
\DeclareMathOperator\op{op}

\newcommand\Z{\mathbb Z}

\newcommand{\s}{\mathfrak s}
\newcommand{\gS}{\mathfrak S}
\renewcommand{\t}{\mathfrak t}
\newcommand{\T}{\mathfrak T}

\newcommand{\uS}{\underline{S}}
\newcommand{\ot}{\otimes}

\newtheoremstyle{definition}
{10pt}
{0pt}
{\normalfont}
{}
{\bfseries}
{}
{ }
{}

\theoremstyle{definition}
\newtheorem{Satz}{Satz}[section]
\newtheorem{Lemma}[Satz]{Lemma}
\newtheorem{Proposition}[Satz]{Proposition}
\newtheorem{Cor}[Satz]{Corollary}
\newtheorem{Theorem}[Satz]{Theorem}
\newtheorem{Definition}[Satz]{Definition}

\newtheorem{Example}[Satz]{Example}

\makeatletter
\renewenvironment{proof}[1][\proofname]{\par
\vspace{-5pt}
\pushQED{\qed}%
\normalfont
\topsep0pt \partopsep0pt 
\trivlist
\item[\hskip\labelsep
\bfseries
#1\@addpunct{:}]\ignorespaces
}{%
\popQED\endtrivlist\@endpefalse
\addvspace{6pt plus 6pt} 
}
\makeatother

\newcounter{num1}

\newcounter{num2}
\newenvironment{thmlist}{\begin{list}{\arabic{num2}.}{\usecounter{num2}
\topsep-5pt \leftmargin23pt \itemindent0pt }}{\end{list}}

\begin{document}

\numberwithin{equation}{section}

\begin{flushright}
   {\sf Hamburger$\;$Beitr\"age$\;$zur$\;$Mathematik$\;$Nr.$\;$664}
\end{flushright}
\vskip 11mm

\begin{center}{\bf \Large 
Hochschild Cohomology and the Modular Group}

\vskip 8mm

{\large Simon Lentner, Svea Nora Mierach,\\  \vskip 1mm Christoph Schweigert, Yorck Sommerh\"auser}
\vskip 9mm
\end{center}

\begin{abstract}
\noindent It has been shown in previous work that the modular group acts
projectively on the center of a factorizable ribbon Hopf algebra. The
center is the zeroth Hochschild cohomology group. In this article, we
extend this projective action of the modular group to an arbitrary
Hochschild cohomology group of a factorizable ribbon Hopf algebra, in
fact up to homotopy even to a projective action on the entire Hochschild
cochain complex.
\end{abstract}

\section*{Introduction}
\enlargethispage{1mm}
An important idea coming from conformal field theory is that modular
categories lead to projective representations of mapping class groups of
surfaces (see \cite{BK}, \cite{G}, \cite{T} and the references cited
therein). At least for certain aspects of this construction, it is not
necessary that the category under consideration is semisimple. For a
particularly simple surface, the torus, the mapping class group is the
homogeneous modular group of two-times-two matrices with integer entries
and determinant one. By applying these ideas in the case of the
representation category of a factorizable ribbon Hopf algebra, which is
not required to be semisimple, we obtain a projective representation of
the homogeneous modular group on the center of this Hopf algebra (see
for example \cite{CW1}, \cite{CW2}, \cite{Ke}, \cite{KL}, \cite{LM}
and~\cite{T}). As the center is the zeroth Hochschild cohomology group
of the Hopf algebra, it is natural to ask whether there is a
corresponding action on the higher cohomology groups. In this article,
we answer this question affirmatively by showing that the modular group
acts, projectively and up to homotopy, even on the entire Hochschild
cochain complex.

The article is organized as follows: In the first section, we briefly
review the Hochschild cohomology of an algebra~$A$ with coefficients in
an $A$-bimodule~$M$, as found for example in~\cite{W}. We then construct in
Proposition~\ref{LeftRight} a particular homotopy between two cochain maps that will be important later for the verification of the defining relations of the modular
group. In the second section, we turn to the case where the algebra~$A$
is a Hopf algebra and introduce a way to modify the bimodule structure
of~$M$ while leaving the Hochschild cohomology groups essentially
unchanged. In the third section, we turn to the case where~$A$ is a
factorizable ribbon Hopf algebra and recall the action of the modular
group on its center. In particular, we introduce the Radford and the
Drinfel'd map. Our treatment here follows largely the exposition
in~\cite{SZ}, to which the reader is referred for references to the
original work. In the fourth section, we take advantage of our
modification of the bimodule structure introduced in the second section
to generalize the Radford and the Drinfel'd map to the Hochschild
cochain complex. In the fifth and final section, we use these maps to 
generalize the action of the modular group on the center to an action 
on all Hochschild cohomology groups of our factorizable ribbon Hopf 
algebra.

We will always work over a base field that is denoted by~$K$, and all
unadorned tensor products are taken over~$K$. The dual of a vector
space~$V$ is denoted by $V^*:=\Hom_K(V,K)$.

The authors would like to thank Sarah Witherspoon for pointing out
References~\cite{FS}, \cite{GK}, \cite{PW} and~\cite{SS} as well as for
further helpful discussions. During the work on this article, the first
and the third author were partially supported by SFB~676 and RTG~1670.

\section{Hochschild Cohomology} \label{Sec:HochCohom}
We begin by briefly recalling the approach to Hochschild cohomology via
the standard resolution. Further details can be found for example
in~\cite[Chap.~IX]{CE} or~\cite[Chap.~9]{W}. We consider an associative
algebra~$A$ over our base field~$K$ and an $A$-bimodule~$M$. As
in~\cite[Chap.~IX, \S~3, p.~167]{CE}, we assume that the left and the
right action of~$A$ on~$M$ become equal when restricted to~$K$, so that
an $A$-bimodule is the same as a module over~$A \ot A^{\op}$. Here~$A^{\op}$ denotes the opposite algebra, in which the product is modified by interchanging the factors.

\begin{Definition}
For an integer~$n>0$, we call $C^n(A,M):=\Hom_K(A^{\ot n},M)$ the space
of cochains, and extend this definition to all integers by setting
$C^0(A,M):=M$ and $C^n(A,M):=0$ for~$n<0$. For $n > 0$ and $i=
0,\dots,n$, we define the coface maps \mbox{$\partial^{n-1}_i\colon
C^{n-1}(A,M) \to C^n(A,M)$} as
\[\partial^{n-1}_i(f)(a_1 \ot \dots \ot a_n) :=
\begin{cases}
a_1.f(a_2\ot\dots\ot a_n) &\text{if } i=0, \\
f(a_1\ot\dots\ot a_{i}a_{i+1}\ot\dots\ot a_n) &\text{if } 0<i<n, \\
f(a_1 \ot \dots \ot a_{n-1}).a_n &\text{if } i=n.
\end{cases}\]
Using these maps, we define the coboundary operator $d^{n-1} \colon
C^{n-1}(A,M)\to C^n(A,M)$, which is also called the differential, as
$d^{n-1} := \sum_{i=0}^n(-1)^i \partial^{n-1}_i$, and extend this
definition to negative numbers by setting~$d^n=0$ for~$n<0$. We then get
a cochain complex
\begin{equation*}
\cdots \xlongrightarrow{d^{-2}} 0 \xlongrightarrow{d^{-1}} M
\xlongrightarrow{d^0} \Hom_K(A,M)\xlongrightarrow{d^1} \Hom_K(A\ot A,M)
\xlongrightarrow{d^2} \cdots
\end{equation*}
that we briefly denote by~$C(A,M)$.
The $n$-th Hochschild cohomology group of the algebra~$A$ with
coefficients in the bimodule~$M$ is defined as the $n$-th cohomology
group of this cochain complex, i.e.,
\[HH^n(A,M):=H^n(C(A,M),d).\]
\end{Definition}

We note that for finite-dimensional separable algebras, and therefore in particular for finite-dimensional semisimple algebras over fields of characteristic zero, the higher Hoch\-schild cohomology groups~$HH^n(A,M)$ for~$n \ge 1$ vanish, as shown for example in~\cite[Chap.~IX, Thm.~7.10, p.~179]{CE}.

The following special cases will be particularly important in the sequel:
\begin{Example} \label{Zero}
For the zeroth Hochschild cohomology group, we find
\begin{align*}
HH^0(A,M)&=\ker(d^0 \colon C^0(A,M)\to C^1(A,M))\\
&=\ker(\partial^0_0-\partial^0_1 \colon C^0(A,M)\to C^1(A,M), m\mapsto
(a\mapsto a.m-m.a))\\
&=\{m\in M\mid a.m=m.a \text{ for all } a\in A\},
\end{align*}
a set that is often called the space of invariants of~$M$, for example
in~\cite[Chap.~IX, \S~4, p.~170]{CE} or~\cite[Sec.~1.1, p.~2]{Ka}. For
$M=A$, where the bimodule structure is given by multiplication, we get
in particular that
\[HH^0(A,A)=Z(A),\]
the center of the algebra~$A$.

For any bimodule~$M$, the dual space $M^*=\Hom_K(M,K)$ is again a
bimodule with respect to the action $(a.\varphi.b)(m)=\varphi(b.m.a)$.
According to the preceding computation, we then have
\begin{align*}
HH^0(A,M^*) &= \{\varphi\in M^*\mid
\varphi(m.a)=\varphi(a.m) \text{ for all } a\in A \text{ and all } m\in M\}.
\end{align*}
\end{Example}

By composition on the left, any bimodule homomorphism~$g \colon M \rightarrow N$ induces a homomorphism
\[g_*: C^n(A,M) \rightarrow C^n(A,N),~f \mapsto g \circ f\]
between the cochain groups, where in general we use a lower star for the map induced by composition on the left and an upper star for the map induced by composition on the right. Because these homomorphisms~$g_*$ commute with the coboundary operators, they can be combined to a cochain map. An element
$c \in Z(A)$ in the center of~$A$ gives rise to two natural choices
for~$g$ on every bimodule~$M$, namely the left and right actions
\[l_c^M \colon M \rightarrow M,~m \mapsto c.m \qquad \text{and} \qquad
r_c^M \colon M \rightarrow M,~m \mapsto m.c.\]
The induced maps on the Hochschild cochain complex are related as follows:
\begin{Proposition} \label{LeftRight}
The cochain maps $(l_c^M)_*$ and $(r_c^M)_*$ are homotopic.
\end{Proposition}
\begin{proof}
For $n \geq 0$, we define $h^{n+1} \colon C^{n+1}(A,M) \to C^{n}(A,M)$ as
\begin{align*}
&h^{n+1}(f)(a_1 \ot \dots \ot a_n) :=
\sum_{j=0}^{n}(-1)^jf(a_1 \ot \dots \ot a_j \ot c \ot a_{j+1} \ot \dots \ot a_n) \\
&\qquad = f(c\ot a_1 \ot \dots \ot a_n)
+ \sum_{j=1}^{n-1} (-1)^j f(a_1 \ot \dots \ot a_j \ot c \ot a_{j+1} \ot
\dots \ot a_n) \\
&\mspace{213mu} + (-1)^n f(a_1 \ot \dots \ot a_n \ot c).
\end{align*}
In particular, we have $h^1(f):=f(c)$. For $n \le 0$, we define
$h^n:=0$, and claim that $h=(h^n)_{n \in \Z}$ is a homotopy between
$(l_c^M)_*$ and $(r_c^M)_*$. To prove this, we have to show that
\[(d^{n-1} \circ h^n + h^{n+1} \circ d^n)(f)(a_1\ot\dots\ot
a_n)=c.f(a_1\ot\dots\ot a_n)-f(a_1\ot\dots\ot a_n).c\]
for all $f \in C^n(A,M) = \Hom_K(A^{\ot n},M)$ and $a_1, \dots, a_n \in A$.
We first show this for the cases involving~$h^1$. For $n=0$, we have as
in Example~\ref{Zero} above that
\begin{align*}
(h^1 \circ d^0)(m) = d^0(m)(c) = c.m-m.c
\end{align*}
for all $m \in M$.
For~$n=1$, we need to consider $f \in \Hom_K(A,M)$ and have
\begin{align*}
&(d^0 \circ h^1 + h^2 \circ d^1)(f)(a)
= a.h^1(f) - h^1(f).a + d^1(f)(c \ot a) - d^1(f)(a \ot c) \\
&= a.f(c) - f(c).a + c.f(a) - f(ca) + f(c).a - a.f(c) + f(ac) - f(a).c
= c.f(a)-f(a).c
\end{align*}
for all $a \in A$, because~$c$ is central.

We now turn to the general case, where $n\geq 2$. For \mbox{$f \in
C^n(A,M)$} and \mbox{$a_1,\dots,a_n\in A$}, we have that
$d^{n-1}(h^n(f))(a_1\ot\dots\ot a_n)$ is given by the sum
\begin{align*}
&d^{n-1}(h^n(f))(a_1\ot\dots\ot a_n) = a_1.h^n(f)(a_2 \ot \dots \ot a_n) \\
&\qquad \quad + \sum_{i=1}^{n-1} (-1)^i
h^n(f)(a_1 \ot \dots \ot a_i a_{i+1} \ot \dots \ot a_n)
+ (-1)^n h^n(f)(a_1\ot\dots\ot a_{n-1}).a_n \\
&\qquad =\sum_{j=1}^{n} (-1)^{j-1} a_1.f(a_2 \ot \dots \ot a_j \ot c \ot
a_{j+1} \ot \dots \ot a_n)\\
&\qquad \quad + t_1 + t_2 +(-1)^n \sum_{j=0}^{n-1} (-1)^j
f(a_1 \ot \dots \ot a_j \ot c \ot a_{j+1}\ot \dots\ot a_{n-1}).a_n,
\end{align*}
where for the second equality we have broken the middle sum into two
terms, namely the term
\begin{align*}
t_1 := & \sum_{0 \le j < i \le n-1} (-1)^{i+j}
f(a_1 \ot \dots \ot a_j \ot c \ot a_{j+1} \ot \dots \ot a_i a_{i+1} \ot
\dots \ot a_n) \\
= & \sum_{0 \le j < i \le n-1} (-1)^{i+j} \partial^n_{i+1}(f)
(a_1 \ot \dots \ot a_j \ot c \ot a_{j+1} \ot \dots \ot a_n) \\
\intertext{and the term}
t_2 := & \sum_{1 \le i \le j \le n-1} (-1)^{i+j}
f(a_1\ot\dots\ot a_ia_{i+1} \ot \dots \ot a_{j+1} \ot c \ot a_{j+2} \ot
\dots \ot a_n) \\
= & \sum_{1 \le i \le j \le n-1} (-1)^{i+j} \partial^n_{i}(f)
(a_1 \ot \dots \ot a_{j+1} \ot c \ot a_{j+2} \ot \dots \ot a_n) \\
= & \sum_{1 \le i < j \le n} (-1)^{i+j-1} \partial^n_{i}(f)
(a_1 \ot \dots \ot a_{j} \ot c \ot a_{j+1} \ot \dots \ot a_n).
\end{align*}

On the other hand, $h^{n+1}(d^n(f))(a_1\ot\dots\ot a_n)$ is given by the sum
\begin{align*}
h^{n+1}(d^n(f))(a_1\ot\dots\ot a_n) &=
\sum_{j=0}^{n} (-1)^j d^n(f)(a_1 \ot \dots \ot a_j \ot c \ot a_{j+1} \ot
\dots \ot a_n)\\
&= \sum_{j=0}^{n} \sum_{i=0}^{n+1} (-1)^{i+j} \partial^n_i(f)
(a_1 \ot \dots \ot a_j \ot c \ot a_{j+1} \ot \dots \ot a_n).
\end{align*}
In the preceding sum, the term for~$i=0$ can be written in the form
\begin{align*}
&\sum_{j=0}^{n} (-1)^j \partial^n_0(f)
(a_1 \ot \dots \ot a_j \ot c \ot a_{j+1} \ot \dots \ot a_n) \\
&\qquad = c.f(a_1\ot\dots\ot a_n) - a_1.h^n(f)(a_2\ot\dots\ot a_n).
\end{align*}
Looking at the term for~$i=n+1$, we get similarly that
\begin{align*}
&\sum_{j=0}^{n} (-1)^j (-1)^{n+1} \partial^n_{n+1}(f)
(a_1 \ot \dots \ot a_j \ot c \ot a_{j+1} \ot \dots \ot a_n) \\
&\qquad = (-1)^{n+1} h^n(f)(a_1\ot\dots\ot a_{n-1}).a_n
- f(a_1\ot\dots\ot a_n).c.
\end{align*}

In the remaining terms, we have~$1 \le i \le n$. The sum of the terms
with $1 \le i < j \le n$ is equal to $- t_2$.
The sum of the terms with $i=j$ is
\begin{align*}
&\sum_{i=1}^{n} f(a_1 \ot \dots \ot a_{i-1} \ot a_i c \ot a_{i+1} \ot
\dots \ot a_n), 
\intertext{while the sum of the terms with $i=j+1$ is}
- &\sum_{j=0}^{n-1} f(a_1 \ot \dots \ot a_j \ot c a_{j+1} \ot a_{j+2}
\ot \dots \ot a_n).
\end{align*}
Because~$c$ is central, these two sums cancel each other. Finally, there
is the sum of the terms with $0 < j+1 < i \le n$, which is equal
to~$-t_1$. Combining all these terms, we find that
\begin{align*}
&h^{n+1}(d^n(f))(a_1\ot\dots\ot a_n)
= c.f(a_1 \ot \dots \ot a_n) - a_1.h^n(f)(a_2 \ot \dots \ot a_n) \\
&\qquad \quad+ (-1)^{n+1} h^n(f)(a_1 \ot \dots \ot a_{n-1}).a_n
- f(a_1\ot\dots\ot a_n).c - t_2 - t_1 \\
&\qquad = c.f(a_1\ot\dots\ot a_n) - d^{n-1}(h^n(f))(a_1\ot\dots\ot a_n)
- f(a_1\ot\dots\ot a_n).c,
\end{align*}
which implies our assertion.
\end{proof}

We note that a similar homotopy for Hochschild homology is described
in~\cite[Par.~1.1.5, p.~10; Exerc.~1.1.2, p.~15]{L}. We also note that
the preceding proposition can be understood from a more abstract and
less computational point of view: In our definition above, we have
realized the Hochschild cohomology groups~$HH^n(A,M)$ as the groups
$\Ext^n_{A\ot A^{\op}}(A,M)$ by using a special resolution of~$A$ as an
$A$-bimodule, or equivalently as an $A\ot A^{\op}$-module, a resolution
that is called the standard resolution in~\cite[Chap.~IX, \S~6,
p.~174f]{CE} and the bar resolution in~\cite[Par.~1.1.12, p.~12]{L}. But in
fact we can work with a general projective resolution
\[A \xleftarrow{\xi} P_0 \xleftarrow{d_1} P_1 \xleftarrow{d_2} P_2
\xleftarrow{d_3} \cdots\]
of~$A$ as an $A$-bimodule, which we briefly denote by~$P$. As already
pointed out above, the fact that~$c$ is central implies that the
maps~$l_c^{P_n}$ and $r_c^{P_n}$ are bimodule homomorphisms.
Because~$\xi$ and the boundary operators~$d_n$ are bimodule
homomorphisms, the maps $l_c^{P_n}$ commute with them, and therefore
lift the left multiplication of~$c$ on~$A$ to the entire resolution:
\begin{center}
\begin{tikzcd} {}
A \arrow{d}{l_c^A} & P_0 \arrow{l}[swap]{\xi}\arrow{d}{l_c^{P_0}} &
P_1 \arrow{l}[swap]{d_1} \arrow{d}{l_c^{P_1}} &
\cdots\arrow{l}[swap]{d_2}\\
A & P_0\arrow{l}[swap]{\xi} & P_1 \arrow{l}[swap]{d_1} & \cdots
\arrow{l}[swap]{d_2}
\end{tikzcd}
\end{center}
Analogously, we can lift the right multiplication of~$c$ on~$A$ to the
entire resolution:
\begin{center}
\begin{tikzcd} {}
A\arrow{d}{r_c^A} & P_0 \arrow{l}[swap]{\xi}\arrow{d}{r_c^{P_0}}
& P_1 \arrow{l}[swap]{d_1} \arrow{d}{r_c^{P_1}} & \cdots
\arrow{l}[swap]{d_2} \\
A & P_0\arrow{l}[swap]{\xi} & P_1\arrow{l}[swap]{d_1} &
\cdots\arrow{l}[swap]{d_2}
\end{tikzcd}
\end{center}

Because~$c$ is central, we have~$r_c^A=l_c^A$. Therefore, the comparison
theorem found in \cite[Chap.~III, Thm.~6.1, p.~87]{ML} or
\cite[Thm.~2.2.6, p.~35]{W} yields that the chain
maps~$l_c^{P}=(l_c^{P_n})$ and~$r_c^{P}=(r_c^{P_n})$ are chain homotopic.

The contravariant functor~$\Hom_{A\ot A^{\op}}(-,M)$ coming from our
bimodule~$M$ turns this homotopy of chain maps into a homotopy of
cochain maps, so that we get that the cochain maps~$(l_c^{P})^*$
and~$(r_c^{P})^*$ are cochain homotopic. But we have~$(l_c^{P})^* =
(l_c^M)_*$: For $f\in\Hom_{A\ot A^{\op}}(P_n,M)$ and $p\in P_n$, we have
\[((l_c^M)_*(f))(p)=l_c^M(f(p)) = c.f(p) = f(c.p)
= f(l_c^{P_n}(p)) = ((l_c^{P_n})^*(f))(p) \]
A similar computation shows that $(r_c^{P})^* = (r_c^M)_*$, which
completes our second, resolution-independent proof of the proposition.

We note that generalizations of this proposition can be found in the
literature, for example in~\cite[Cor.~1.3, p.~709]{SS}. However, we will
only need the above form of the proposition in the sequel.

\section{Hochschild Cohomology of Hopf Algebras} \label{Sec:HochschHopf}
We now turn to the case where the algebra~$A$ is a Hopf algebra. We will
denote the coproduct of~$A$ by~$\Delta$, its counit by~$\varepsilon$,
and its antipode by~$S$. For the coproduct of $a \in A$, we will use
Heyneman-Sweedler notation in the form $\Delta(a) = a_{(1)} \ot a_{(2)}$.

Because~$A$ is a Hopf algebra, every $A$-bimodule~$M$ can be considered
as a right $A$-module via the right adjoint action
\[\ad \colon M \ot A \rightarrow M,~m \ot a \mapsto \ad(m \ot a),\]
which is defined as $\ad(m \ot a) := S(a_{(1)}).m.a_{(2)}$. We
denote~$M$ by~$M_{\ad}$ if it is considered as a right $A$-module in
this way.

In general, a right $A$-module~$N$ becomes an $A$-bimodule with respect
to the trivial left action, i.e., the action defined as~$a.n :=
\varepsilon(a) n$. We denote~$N$
by~$\prescript{}{\varepsilon}N$ if it is considered as a bimodule in
this way. By combining the two operations, we can associate with an
\mbox{$A$-bimodule~$M$} the
$A$-bimodule~$\prescript{}{\varepsilon}M_{\ad}:=\prescript{}{\varepsilon}(M_{\ad})$.
As it turns out, the Hochschild cochain complexes determined by these
two bimodules are isomorphic:

\begin{Proposition} \label{IsoCompl}
The maps $\Omega^n \colon C^n(A, \prescript{}{\varepsilon}M_{\ad}) \to
C^n(A,M)$ defined via the formula
\[\Omega^n(f)(a_1\ot\dots\ot a_n) =
a_{1{(1)}} \dots a_{n{(1)}}.f(a_{1{(2)}} \ot \dots \ot a_{n{(2)}})\]
give rise to an isomorphism~$\Omega = (\Omega^n)$ between the Hochschild
cochain complex of~$\prescript{}{\varepsilon}M_{\ad}$ and the Hochschild
cochain complex of~$M$.
\end{Proposition}
\begin{proof}
We first note that $\Omega^n$ is bijective with inverse
\[(\Omega^n)^{-1}(f)(a_1\ot \dots \ot a_n) :=
S(a_{n{(1)}}) \dots S(a_{1{(1)}}).f(a_{1{(2)}} \ot \dots \ot a_{n{(2)}}),\]
because for $f \in C^n(A, \prescript{}{\varepsilon}M_{\ad}) =
\Hom_K(A^{\ot n}, \prescript{}{\varepsilon}M_{\ad})$, we have
\begin{align*}
&(\Omega^n)^{-1}(\Omega^n(f))(a_1 \ot \dots \ot a_n) =
S(a_{n(1)}) \dots S(a_{1(1)}).\Omega^n(f)(a_{1(2)}\ot\dots\ot a_{n(2)}) \\
&\qquad = S(a_{n(1)}) \dots S(a_{1(1)}) a_{1(2)} \dots a_{n(2)}.f(a_{1(3)} \ot
\dots \ot a_{n(3)})
=f(a_1\ot\dots\ot a_n),
\end{align*}
and the relation $\Omega^n \circ (\Omega^n)^{-1} = \id_{C^n(A,M)}$ follows
analogously.

For $f \in C^{n-1}(A, \prescript{}{\varepsilon}M_{\ad})$
and~$a_1,\dots,a_n\in A$, we have on the one hand
\begin{align*}
&d^{n-1}(\Omega^{n-1}(f))(a_1 \ot \dots \ot a_n) =
a_1.\Omega^{n-1}(f)(a_2 \ot \dots \ot a_n) \\
&\qquad \quad + \sum_{i=1}^{n-1}(-1)^i\Omega^{n-1}(f)(a_1 \ot \dots \ot a_i
a_{i+1} \ot \dots \ot a_n)\\
&\qquad \quad + (-1)^n (\Omega^{n-1}(f)(a_1 \ot \dots \ot a_{n-1})).a_n\\
&\qquad = a_1 a_{2{(1)}} \dots a_{n{(1)}}.f(a_{2{(2)}} \ot \dots \ot a_{n{(2)}}) \\
&\qquad \quad + \sum_{i=1}^{n-1}(-1)^i
a_{1(1)} \dots (a_{i} a_{i+1})_{(1)} \dots a_{n(1)}.
f(a_{1(2)} \ot \dots \ot (a_{i} a_{i+1})_{(2)} \ot \dots \ot a_{n(2)})\\
&\qquad \quad + (-1)^n a_{1(1)} \dots a_{n-1(1)}.f(a_{1(2)} \ot \dots \ot
a_{n-1(2)}).a_n
\end{align*}
and on the other hand
\begin{align*}
&\Omega^n(d^{n-1}(f))(a_1 \ot \dots \ot a_n) =
a_{1{(1)}} \dots a_{n{(1)}}.d^{n-1}(f)(a_{1{(2)}} \ot \dots \ot
a_{n{(2)}})\\
&\qquad = a_{1{(1)}} \dots a_{n{(1)}}.\varepsilon(a_{1{(2)}}) f(a_{2{(2)}} \ot
\dots \ot a_{n{(2)}})\\
&\qquad \quad + \sum_{i=1}^{n-1} (-1)^i a_{1(1)} \dots a_{n(1)}.f(a_{1(2)} \ot
\dots \ot a_{i(2)} a_{i+1(2)} \ot \dots \ot a_{n(2)})\\
&\qquad \quad + (-1)^n a_{1{(1)}} \dots a_{n{(1)}}.
\left(S(a_{n(2)}).f(a_{1(2)} \ot \dots \ot a_{n-1(2)}).a_{n(3)} \right),
\end{align*}
where we have used for the last summand that
\begin{align*}
\ad(f(a_{1} \ot \dots \ot a_{n-1}) \ot a_{n}) =
S(a_{n(1)}).f(a_{1} \ot \dots \ot a_{n-1}).a_{n(2)}
\end{align*}
according to the definition of the right adjoint action. Because both
expressions agree, $\Omega$~is a cochain map, which establishes our
assertion.
\end{proof}

If the antipode of~$A$ is bijective, the coopposite Hopf
algebra~$A^{\cop}$, in which the product remains unaltered, but the
coproduct is modified by interchanging the tensor factors, is a Hopf
algebra, and its antipode is the inverse of the antipode of~$A$. For an
$A$-bimodule~$M$, we denote the right adjoint action that arises from
this Hopf algebra structure by~$\cad$; in terms of the original
structure elements, this action is given by the formula
\[\cad(m \ot a) := S^{-1}(a_{(2)}).m.a_{(1)}.\]
If we apply the preceding proposition to this situation, we obtain the
following corollary:
\begin{Cor} \label{IsoComplCor}
If the antipode of~$A$ is bijective, the maps $\Omega'^n \colon C^n(A,
\prescript{}{\varepsilon}M_{\cad}) \to C^n(A,M)$ defined via the formula
\[\Omega'^n(f)(a_1\ot\dots\ot a_n) =
a_{1{(2)}} \dots a_{n{(2)}}.f(a_{1{(1)}} \ot \dots \ot a_{n{(1)}})\]
give rise to an isomorphism~$\Omega' = (\Omega'^n)$ between the
Hochschild cochain complex of~$\prescript{}{\varepsilon}M_{\cad}$ and
the Hochschild cochain complex of~$M$.
\end{Cor}
We record that $\Omega'^n$ is bijective with inverse
\[(\Omega'^n)^{-1}(f)(a_1\ot \dots \ot a_n) :=
S^{-1}(a_{n{(2)}}) \dots S^{-1}(a_{1{(2)}}).f(a_{1{(1)}} \ot \dots \ot
a_{n{(1)}}),\]
as we had seen in the proof of our proposition.

Proposition~\ref{IsoCompl} generalizes a result found in \cite[Sec.~1,
p.~2862f]{FS}. We note that further results related to this proposition
can be found in the literature: In the case where the Hopf algebra is a
group ring, the argument is contained in~\cite[\S~5, p.~60f]{EM}, one of
the foundational articles for group cohomology. A homology version of
the proposition can be found in~\cite[Prop.~(2.4), p.~488]{FT}, at least
in the case where the bimodule is the underlying algebra. Similar
statements for cohomology appear in~\cite[Par.~5.5, p.~197]{GK}
and~\cite[Lem.~12, p.~591]{PW}. These last two references, however,
rather state a combination of Proposition~\ref{IsoCompl} with the
following lemma:
\begin{Lemma}
For a right $A$-module~$N$, we have $HH^n(A, \prescript{}{\varepsilon}N)
\cong \Ext_A^n(K, N)$, where the base field~$K$ is given the trivial
right $A$-module structure via the counit~$\varepsilon$.
\end{Lemma}
\begin{proof}
If $P=(P_n)$ is a projective resolution of~$A$ as a left $A \ot
A^{\op}$-module, we know from \cite[Chap.~X, Thm.~2.1, p.~185]{CE} that
$K \ot_A P := (K \ot_A P_n)$ is a projective resolution of~$K$ as a
right $A$-module. Therefore $\Ext_A^n(K, N)$ is the $n$-th cohomology group
of the cochain complex formed by the cochain groups $\Hom_A(K \ot_A P_n,
N)$. But
the cochain map
\[\Hom_{A \ot A^{\op}}(P_n, \prescript{}{\varepsilon}N) \to \Hom_A(K
\ot_A P_n, N),~f \mapsto
(\lambda \ot p \mapsto \lambda f(p)) \]
with inverse
\[\Hom_A(K \ot_A P_n, N) \to \Hom_{A \ot A^{\op}}(P_n,
\prescript{}{\varepsilon}N),~g \mapsto
(p \mapsto g(1_K \ot p)) \]
shows that this complex is isomorphic to the cochain complex of the
cochain groups
$\Hom_{A \ot A^{\op}}(P_n, \prescript{}{\varepsilon}N)$, whose
cohomology groups are
$\Ext^n_{A \ot A^{\op}}(A, \prescript{}{\varepsilon}N) = HH^n(A,
\prescript{}{\varepsilon}N)$.
\end{proof}

\section{The Action on the Center} \label{Sec:ActCent}
We now turn to the case of a factorizable ribbon Hopf algebra~$A$ with
R-matrix~$R$ and ribbon element~$v$. Even though the R-matrix is in
general not a pure tensor, we use the notation $R = R_1 \ot R_2$. If
$\tau$ denotes the flip map, we therefore have
$\tau(R) = R_2 \ot R_1$. This element in turn can be used to
introduce the monodromy matrix $Q := \tau(R) R$, and as for the R-matrix,
we write $Q=Q_1\ot Q_2$. An important role will be played by the
Drinfel'd and Radford map, which are defined as follows:

\begin{Definition} \label{DrinfRadfMap}
We call the map
\[\bar{\Phi}\colon A^*\to A,\ \varphi\mapsto\varphi(Q_1)Q_2\]
the Drinfel'd map, and define the subalgebra 
\[\bar{C}(A) := \{\varphi\in A^*\mid \varphi(bS^{-2}(a)) = \varphi(ab)
\text{ for all } a,b\in A\},\]
whose elements we call generalized class functions. With the help of a nonzero right integral $\rho\in A^*$, we introduce
the Radford map
\[\iota\colon A\to A^*,\ a\mapsto\rho_{(1)}(a)\rho_{(2)}.\]
\end{Definition}

By definition, $A$~is factorizable if and only if~$\bar{\Phi}$ is
bijective, which implies in particular that~$A$ is finite-dimensional.
The basic properties of the Drinfel'd map can be found
in~\cite[Par.~3.2, p.~26]{SZ}, and the basic properties of the Radford
map can also be found there, namely in~\cite[Par.~4.1, p.~35]{SZ}. In
particular, the Drinfel'd map restricts to an algebra isomorphism
from~$\bar{C}(A)$ to~$Z(A)$, the center of~$A$, while the Radford map
restricts to a $K$-linear isomorphism from~$Z(A)$ to~$\bar{C}(A)$. 
A consequence of this last fact is that~\mbox{$\rho = \iota(1_A) \in \bar{C}(A)$}, 
which is a special case of a general result found 
in~\cite[Thm.~10.5.4, p.~307]{R} that arises when combined 
with~\cite[Prop.~12.4.2, p.~405]{R}.

Following~\cite[Par.~4.1, p.~35]{SZ}, we introduce the
endomorphism~$\gS:=S\circ\bar{\Phi}\circ\iota$ of~$A$, where as
before~$S$ denotes the antipode of~$A$. For $a\in A$, we have explicitly
\begin{align*}
\gS(a) &= S(\bar{\Phi}(\rho_{(1)}(a)\rho_{(2)})) =
S(\rho_{(1)}(a)\rho_{(2)}(Q_1)Q_2)
= \rho(aQ_1) \, S(Q_2),
\end{align*}
or $\gS(a) = \rho(a R_2 R'_1) \, S(R_1 R'_2)$ if we insert the definition
of the monodromy matrix by using a second copy~$R'$ of the R-matrix.

As in~\cite[Par.~4.3, p.~37]{SZ}, we introduce a second such map, namely
the multiplication
\[\T\colon A\to A, \ a \mapsto va\]
with the ribbon element~$v \in A$. The endomorphisms~$\gS$ and~$\T$ will
be used to encode the action of the two generators of the modular group
described below.

We will need a third endomorphism of~$A$, namely the antipode of the
transmutation of~$A$.
The transmutation of a quasitriangular Hopf algebra was described by
S.~Majid in several articles, among them~\cite{M1}, and is discussed in
his monograph~\cite{M2}. It has the same underlying vector space as~$A$,
in fact even the same algebra structure. In the version that we are
using, the antipode $\uS$ of the transmutation is given by
\[\uS(a) = S(S(R_{1(1)}) a R_{1(2)}) R_2.\]
This variant arises from the one given in~\cite[Ex.~9.4.9, p.~504]{M2}
by replacing~$A$ with~$A^{\op\cop}$. If $u:=S(R_2)R_1$ is the Drinfel'd
element of $A$, then the element
$S(u) = R_1 S(R_2)$ is the Drinfel'd element of $A^{\op\cop}$.
Therefore, the alternative form of~$\uS$ given in~\cite[Eq.~(9.42),
p.~507]{M2} becomes in our case
\[\uS(a) = R_1 S(a) S(R_2) S(u^{-1}).\]

These three endomorphisms are related as follows:
\begin{Proposition} \label{ModRel}
The maps $\gS$ and $\T$ satisfy the relations
\[\gS \circ \T \circ \gS = \rho(v) \; \T^{-1} \circ \gS \circ \T^{-1}
\quad \text{and} \quad
\gS^2 = (\rho\ot\rho)(Q) \; \uS^{-1}.\]
\end{Proposition}
\begin{proof}
A proof of the first relation can be found in~\cite[Prop.~4.3,
p.~37]{SZ}. To prove the second relation, we use four copies $R,R',R''$
and~$R'''$ of the R-matrix. Because $\rho\in\bar{C}(A)$, the map~$\gS$
is alternatively given by
\[\gS(a) = \rho(a R_2 R'_1) \, S(R'_2) S(R_1) = \rho(a R_2 S^{-1}(R'_1)) \, R'_2 S(R_1)
= \rho(S(R'_1) a R_2) \, R'_2 S(R_1), \]
where we have used the fact $(S \ot S)(R)=R$ proved in
\cite[Prop.~10.1.8, p.~180]{M}.
From~\cite[Prop.~4.1, p.~35]{SZ}, we know that~$\gS$ is $A$-linear with
respect to the right adjoint action. Therefore, we have
\[\uS(\gS(a)) = S(\gS(S(R''_{1(1)}) a R''_{1(2)})) R''_2
= \rho(S(R'_1) S(R''_{1(1)}) a R''_{1(2)} R_2) \, S(R'_2 S(R_1)) R''_2.\]
If we use one of the axioms for the R-matrix, namely~\cite[Eq.~10.1.6,
p.~180]{M}, and the fact that the antipode is antimultiplicative, this
equation can be rewritten in the form
\begin{align*}
\uS(\gS(a)) &= \rho(S(R'_1) S(R'''_1) a R''_1 R_2) \, S^2(R_1) S(R'_2)
R'''_2 R''_2 \\
&= \rho(R'_1 S(R'''_1) a R''_1 R_2) \, S^2(R_1) R'_2 R'''_2 R''_2.
\end{align*}
Another fact proved in \cite[Prop.~10.1.8, p.~180]{M} is that
$(S \ot \id)(R)=R^{-1}$, so that this equation reduces to
\[\uS(\gS(a)) = \rho(a R''_1 R_2) \, S^2(R_1) R''_2
=(\rho\ot\rho)(Q) \, \gS^{-1}(a),\]
where the last step follows from~\cite[Prop.~4.2, p.~36]{SZ}. Our claim
is a minor rearrangement of this equation.
\end{proof}

We would like to emphasize that this proposition is not new; rather, it
is a variant of~\cite[Thm.~4.4, p.~523]{LM}. We also note that it follows directly from another elementary property of R-matrices also proved in~\cite[Prop.~10.1.8, p.~180]{M} that~$\uS$ agrees with the ordinary antipode~$S$ on the center of~$A$, and because the square of the
antipode is given by conjugation with the Drinfel'd element~$u$, as shown 
in~\cite[Prop.~10.1.4, p.~179]{M}, we have $S=S^{-1}$ on the center. Therefore, the second relation in the previous
proposition generalizes \cite[Cor.~4.2, p.~37]{SZ}.

The fact that the square of~$\uS$ restricts to the identity on the
center can also be seen from the fact that, in general, it is given by
the right adjoint action of our ribbon element:

\begin{Lemma} \label{SqRib}
For all $a \in A$, we have $\uS^2(a) = \ad(a \ot v)$.
\end{Lemma}
\begin{proof}
With the help of the alternative form of~$\uS$, we get
\begin{align*}
\uS^2(a) &= \uS(R_1 S(a) S(R_2) S(u^{-1}))
= R'_1 S(R_1 S(a) S(R_2) S(u^{-1})) S(R'_2) S(u^{-1})\\
&= R'_1 S^2(u^{-1}) S^2(R_2) S^2(a)S(R_1) S(R'_2) S(u^{-1})
= R'_1 u^{-1} S^2(R_2a) S(u^{-1} R'_2 R_1).
\end{align*}
Using the definition of the monodromy matrix~$Q$, the basic
properties of ribbon elements found in~\cite[Par.~4.3, p.~37]{SZ} and the above-mentioned fact that the square of the antipode is given by conjugation with the Drinfel'd element~$u$, this becomes
\begin{align*}
\uS^2(a) &= Q_2au^{-1}S(u^{-1}Q_1) = S^2(Q_2)au^{-1}S(u^{-1}S^2(Q_1))
= S^2(Q_2)au^{-1}S(Q_1u^{-1})\\
&= S^2(Q_2)au^{-1}S(u^{-1})S(Q_1) = S^2(Q_2)av^2S(Q_1)
= S(Q_1)av^2Q_2=S(v_{(1)})av_{(2)}
\end{align*}
as asserted.
\end{proof}
We note that this equation is stated in \cite[Eq.~(2.60), p.~370]{Ke},
at least in the case of Drinfel'd doubles. A version in the framework of
coends can be found in \cite[Cor.~3.10, p.~306]{Ly}.

The proposition above implies that the (homogeneous) modular group
\[\SL(2,\Z)=\{M\in\GL(2,\Z)\mid\det(M)=1\}\]
acts projectively on the center of~$A$. The modular group is generated
by the two elements $\s:=\big(\begin{smallmatrix}
0 &-1\\
1 & 0
\end{smallmatrix}\big)$ and $\t:=\big(\begin{smallmatrix}
1 & 1\\
0 & 1
\end{smallmatrix}\big)$, which satisfy the relations
\[\s^4=1 \quad\text{and}\quad \s \t \s = \t^{-1} \s \t^{-1} ,\]
and these relations are defining, as shown for example
in~\cite[Thm.~3.2.3.2, p.~97]{FR}, \cite[Thm.~A.2, p.~312]{KT}
or~\cite[Sec.~II.1, Thm.~8, p.~53]{Ma}.

If we denote the projective space associated to~$Z(A)$ by~$P(Z(A))$ and
the automorphisms of this projective space arising from~$\gS$ and~$\T$
by~$P(\gS)$ and~$P(\T)$, the above proposition implies immediately the
following fact:
\begin{Cor} \label{ProjRepCent}
There is a unique homomorphism from $\SL(2,\Z)$ to $\PGL(Z(A))$ that
maps~$\s$ to~$P(\gS)$ and~$\t$ to~$P(\T)$.
\end{Cor}

This result holds for any ribbon element~$v \in A$ and any nonzero right
integral $\rho \in A^*$. As shown in~\cite[Cor.~12.4.4, p.~407]{R}, we
have~$\rho(v) \neq 0$; this is obviously also a consequence of the
proposition above. Because right integrals are only unique up to scalar
multiples, we can choose a right integral that satisfies~$\rho(v)=1$;
following \cite[Def.~4.4, p.~39]{SZ}, we call such a right integral
ribbon-normalized with respect to~$v$. If we use a ribbon-normalized
right integral, the proposition above shows that the action of the
modular group on the center is linear, and not only projective, if and
only if $(\rho\ot\rho)(Q) = \pm 1$. By \cite[Lem.~4.4, p.~39]{SZ}, this
condition is equivalent to the condition~$\rho(v^{-1}) = \pm 1$.

\section{The Radford and the Drinfel'd Map for Complexes}
\label{Sec:RadfDrinf}
We remain in the situation described in Section~\ref{Sec:ActCent} and
consider a factorizable ribbon Hopf algebra~$A$ with R-matrix~$R$ and
ribbon element~$v$. Our first goal is to generalize the Radford map, the
Drinfel'd map and the antipode to cochain maps of Hochschild cochain
complexes. We begin with the Radford map, for which this is particularly
easy.

By~$A_{S^{-2}}$, we denote~$A$ considered as an $A$-bimodule with the
left action given by multiplication, but the right action modified via
the square of the inverse antipode, so that the right action is given by
$b.a := bS^{-2}(a)$ for $a,b\in A$. As explained in Example~\ref{Zero},
we then have
\begin{align*}
HH^0(A,(A_{S^{-2}})^*)&=
\{\varphi\in A^*\mid \varphi(bS^{-2}(a))=\varphi(ab) \text{ for all }
a,b\in A\}
= \bar{C}(A),
\end{align*}
the algebra of generalized class functions introduced in
Definition~\ref{DrinfRadfMap}. The bimodule~$A_{S^{-2}}$ is related to
the Radford map in the following way:

\begin{Proposition} \label{RadfBimod}
The Radford map~$\iota$ is a bimodule isomorphism from~$A$
to~$(A_{S^{-2}})^*$.
\end{Proposition}
\begin{proof}
By \cite[Thm.~2.1.3, p.~18]{M}, $A$~is a Frobenius algebra with
Frobenius homomorphism~$\rho$, so that $\iota$ is bijective. It is a
bimodule homomorphism because
\begin{align*}
\iota(a_1 a a_2)(b) = \rho(a_1 a a_2 b) = \rho(a a_2 b S^{-2}(a_1))
= \iota(a)(a_2 b S^{-2}(a_1)) = (a_1.\iota(a).a_2)(b)
\end{align*}
for all $a, a_1, a_2, b\in A$, where the second equality holds because
$\rho \in \bar{C}(A)$, a fact already pointed out in Section~\ref{Sec:ActCent}.
\end{proof}

Because bimodule isomorphisms induce isomorphisms between the
corresponding Hoch\-schild cochain complexes, this proposition enables
us to generalize the Radford map to a cochain map as follows:
\begin{Definition} \label{iota1}
We define the Radford map for Hochschild cochain complexes as the
cochain map from~$C(A,A)$
to~$C(A,(A_{S^{-2}})^*)$ with components
\[\iota^n \colon C^n(A,A) \to C^n(A,(A_{S^{-2}})^*),~f \mapsto \iota \circ f.\]
In other words, we set $\iota^n := \iota_*$, the composition
with~$\iota$ on the left.
\end{Definition}

In order to compare this definition with the treatment of the Drinfel'd
map and the antipode below, it will be important to relate this cochain
map to another one defined between different cochain complexes. From
Proposition~\ref{IsoCompl}, we get a cochain map $\Omega = (\Omega^n)$
from the cochain complex $C^n(A, \prescript{}{\varepsilon}A_{\ad})$ to
the cochain complex~$C^n(A,A)$, but also a cochain map from the cochain
complex~$C^n(A, \prescript{}{\varepsilon}((A_{S^{-2}})^*)_{\ad})$ to the
cochain complex~$ C^n(A,(A_{S^{-2}})^*)$, which we denote by $\Omega'' =
(\Omega''^n)$. The bimodule
$\prescript{}{\varepsilon}((A_{S^{-2}})^*)_{\ad}$ admits a slightly
simpler description: For
$\varphi \in (A_{S^{-2}})^*$, $a\in A$ and $b \in A_{S^{-2}}$, we have
\begin{align*}
\ad(\varphi \ot a)(b) &= (S(a_{(1)}).\varphi.a_{(2)})(b) =
\varphi(a_{(2)}.b.S(a_{(1)})) \\
&= \varphi(a_{(2)} b S^{-1}(a_{(1)})) 
= \varphi_{(1)}(a_{(2)}) \varphi_{(2)}(b) \varphi_{(3)}(S^{-1}(a_{(1)})),
\end{align*}
which shows that the right adjoint action
in~$\prescript{}{\varepsilon}((A_{S^{-2}})^*)_{\ad}$ coincides with the
right coadjoint action of the coopposite Hopf algebra~$A^{\cop}$, which
we denote by
\[\coad \colon A^* \ot A \to A^*,~\varphi \ot a \mapsto
\varphi_{(1)}(a_{(2)}) \varphi_{(3)}(S^{-1}(a_{(1)})) \, \varphi_{(2)}.\]
In other words, we have
$\prescript{}{\varepsilon}((A_{S^{-2}})^*)_{\ad} =
\prescript{}{\varepsilon}(A^*)_{\coad}$. The Radford map now relates the
two isomorphisms~$\Omega$ and~$\Omega''$ as follows:

\begin{Lemma} \label{iota2} The diagram
\begin{center}
\begin{tikzcd} {}
C^n(A, \prescript{}{\varepsilon}(A^*)_{\coad}) \arrow{r}{\Omega''^n} &
C^n(A,(A_{S^{-2}})^*)\\
C^n(A,\prescript{}{\varepsilon}{A}_{\ad}) \arrow{u}{\iota_*}
\arrow{r}{\Omega^n} &
C^n(A,A) \arrow{u}{\iota_*}
\end{tikzcd}
\end{center}
commutes.
\end{Lemma}
\begin{proof}
We first note that it follows from Proposition~\ref{RadfBimod} above
that the Radford map~$\iota$ is also a bimodule isomorphism
from~$\prescript{}{\varepsilon}{A}_{\ad}$
to~$\prescript{}{\varepsilon}((A_{S^{-2}})^*)_{\ad} =
\prescript{}{\varepsilon}(A^*)_{\coad}$, so that the map on the left is
well-defined. For $f\in C^n(A,\prescript{}{\varepsilon}{A}_{\ad}) =
\Hom_K(A^{\ot n},\prescript{}{\varepsilon}{A}_{\ad})$ and
$a_1,\dots,a_n,b\in A$, we now have on the one hand
\begin{align*}
\iota(\Omega^n(f)(a_1 \ot \dots \ot a_n))(b)
&= \iota(a_{1(1)} \dots a_{n(1)} f(a_{1(2)} \ot \dots \ot a_{n(2)}))(b)\\
&= \rho(a_{1(1)} \dots a_{n(1)} f(a_{1(2)} \ot \dots \ot a_{n(2)}) b).
\end{align*}

On the other hand, we have
\begin{align*}
((\Omega''^n\circ\iota_*)(f)(a_1\ot\dots\ot a_n))(b)
&= (a_{1(1)} \dots a_{n(1)}.\iota_*(f)(a_{1(2)} \ot \dots \ot
a_{n(2)}))(b) \\
&= \iota_*(f)(a_{1(2)} \ot \dots \ot a_{n(2)})(b.a_{1(1)} \dots a_{n(1)}) \\
&= \iota(f(a_{1(2)} \ot \dots \ot a_{n(2)}))(bS^{-2}(a_{1(1)} \dots
a_{n(1)})) \\
&= \rho(f(a_{1(2)} \ot \dots \ot a_{n(2)}) b S^{-2}(a_{1(1)} \dots
a_{n(1)})).
\end{align*}
Since $\rho \in \bar{C}(A)$, these expressions are equal.
\end{proof}

To generalize the Drinfel'd map to a cochain map between Hochschild
cochain complexes, we first recall from \cite[Par.~3.2, p.~26]{SZ} that
the Drinfel'd map $\bar{\Phi}$ is a bimodule isomorphism between
$\prescript{}{\varepsilon}{(A^*)}_{\coad}$ and
$\prescript{}{\varepsilon}A_{\cad}$, so that we obtain an isomorphism of
cochain complexes
\[\bar{\Phi}_*\colon C^n(A,\prescript{}{\varepsilon}{(A^*)}_{\coad})
\to C^n(A,\prescript{}{\varepsilon}{A}_{\cad}),~f \mapsto \bar{\Phi}
\circ f\]
by composing with~$\bar{\Phi}$ on the left. Now the isomorphism
$\Omega'$ from Corollary~\ref{IsoComplCor} enables us to obtain a
cochain map between the original cochain complexes:
\begin{Definition}
We define the Drinfel'd map for Hochschild cochain complexes as the cochain map
from~$C(A,(A_{S^{-2}})^*)$
to~$C(A,A)$ with components
$\bar{\Phi}^n := \Omega'^n \circ \bar{\Phi}_* \circ (\Omega''^n)^{-1}$. In
other words, it is the unique cochain map whose components make the diagram
\begin{center}
\begin{tikzcd}\label{GenDrinf} {}
C^n(A,\prescript{}{\varepsilon}{(A^*)}_{\coad}) \arrow{d}{\bar{\Phi}_*}
\arrow{r}{\Omega''^n} &
C^n(A,(A_{S^{-2}})^*) \arrow{d}{\bar{\Phi}^n} \\
C^n(A,\prescript{}{\varepsilon}{A}_{\cad}) \arrow{r}{\Omega'^n} &
C^n(A,A)
\end{tikzcd}
\end{center}
commutative.
\end{Definition}

With the help of the monodromy matrix~$Q$, the map $\bar{\Phi}^n$ can be
calculated explicitly:
For $f \in C^n(A,(A_{S^{-2}})^*) = \Hom_K(A^{\ot n},(A_{S^{-2}})^*)$ and
$a_1, \dots, a_n \in A$, we have
\begin{align*}
\bar{\Phi}^n(f)(a_1\ot\dots\ot a_n)
&= (\Omega'^n \circ \bar{\Phi}_* \circ (\Omega''^n)^{-1})(f)(a_1 \ot \dots
\ot a_n) \\
&= a_{1(2)} \dots a_{n(2)} \bar{\Phi}((\Omega''^n)^{-1}(f)(a_{1(1)} \ot
\dots \ot a_{n(1)}))\\
&= a_{1(3)} \dots a_{n(3)} \bar{\Phi}(S(a_{n(1)}) \dots
S(a_{1(1)}).f(a_{1(2)} \ot \dots \ot a_{n(2)})) \\
&= a_{1(3)} \dots a_{n(3)}
(S(a_{n(1)}) \dots S(a_{1(1)}).f(a_{1(2)} \ot \dots \ot
a_{n(2)}))(Q_1)Q_2 \\
&= a_{1(3)} \dots a_{n(3)}
f(a_{1(2)} \ot \dots \ot a_{n(2)}) (Q_1.(S(a_{n(1)}) \dots
S(a_{1(1)}))) Q_2 \\
&= f(a_{1(2)}\ot\dots\ot a_{n(2)})(Q_1S^{-1}(a_{n(1)})\dots
S^{-1}(a_{1(1)})) \,
a_{1(3)} \dots a_{n(3)} Q_2.
\end{align*}

In a similar way, we can generalize the antipode~$S$ to a cochain map
between Hochschild cochain complexes: Since we have
\[S(\cad(b \ot a)) = S(S^{-1}(a_{(2)})ba_{(1)}) = S(a_{(1)})S(b)a_{(2)}
= \ad(S(b) \ot a),\]
the antipode is a bimodule isomorphism from
$\prescript{}{\varepsilon}{A}_{\cad}$ to
$\prescript{}{\varepsilon}{A}_{\ad}$. Composition with~$S$ therefore
yields an isomorphism
\[S_*\colon C^n(A,\prescript{}{\varepsilon}{A}_{\cad})
\to C^n(A,\prescript{}{\varepsilon}{A}_{\ad}),~f \mapsto S \circ f\]
of cochain complexes. Now the isomorphisms~$\Omega$ and~$\Omega'$ from
Section~\ref{Sec:HochschHopf} enable us to obtain a cochain map between
the original cochain complexes:
\begin{Definition} \label{GenAntip}
We define the antipode map for Hochschild cochain complexes as the cochain map from
$C(A, A)$ to itself with components $S^n := \Omega^n \circ S_* \circ
(\Omega'^n)^{-1}$. In other words, it is the unique cochain map whose
components make the diagram
\begin{center}
\begin{tikzcd} {}
C^n(A, \prescript{}{\varepsilon}{A}_{\cad}) \arrow{d}{S_*}
\arrow{r}{\Omega'^n} &
C^n(A, A) \arrow{d}{S^n} \\
C^n(A, \prescript{}{\varepsilon}{A}_{\ad}) \arrow{r}{\Omega^n} & C^n(A,
A)
\end{tikzcd}
\end{center}
commutative.
\end{Definition}

As in the case of the Drinfel'd map, there is an explicit expression for
the antipode map for cochain complexes: For $f\in C^n(A, A) =
\Hom_K(A^{\ot n}, A)$ and $a_1, \dots, a_n \in A$, we have
\begin{align*}
S^n(f)(a_1 \ot \dots \ot a_n)
&= (\Omega^n \circ S_* \circ (\Omega'^n)^{-1})(f)(a_1 \ot \dots \ot a_n) \\
&= a_{1(1)} \dots a_{n(1)}S((\Omega'^n)^{-1}(f)(a_{1(2)}\ot\dots\ot
a_{n(2)}))\\
&= a_{1(1)}\dots a_{n(1)}S(S^{-1}(a_{n(3)})\dots
S^{-1}(a_{1(3)})f(a_{1(2)}\ot\dots\ot a_{n(2)}))\\
&= a_{1(1)}\dots a_{n(1)}S(f(a_{1(2)}\ot\dots\ot
a_{n(2)}))a_{1(3)}\dots a_{n(3)}.
\end{align*}

\section{The Action on the Hochschild Cochain Complex} \label{Sec:ActHoch}
We still remain in the situation described in Section~\ref{Sec:ActCent}
and Section~\ref{Sec:RadfDrinf}. Our goal is to use the Radford map, the
Drinfel'd map and the antipode map for Hochschild cochain complexes
introduced in Section~\ref{Sec:RadfDrinf} in order to construct a
projective action of the modular group~$\SL(2,\Z)$ on each Hochschild
cohomology group~$HH^n(A,A)$ in such a way that the action on the zeroth
Hochschild cohomology group, which is, as we saw in Example~\ref{Zero},
equal to the center~$Z(A)$, coincides with the action described in
Section~\ref{Sec:ActCent}. Up to homotopy, we will in fact construct a
projective action of the modular group on the entire Hochschild cochain
complex.

To define a projective representation of the modular group, we have to
specify the images of the generators~$\s$ and~$\t$ introduced in
Section~\ref{Sec:ActCent} and prove that they satisfy the defining
relations stated there. For the first generator~$\s$, we use the same
approach as in Section~\ref{Sec:ActCent} and map it to the composition
of the Radford map, the Drinfel'd map and the antipode:
\begin{Definition}
We define $\gS^n \colon C^n(A, A) \to C^n(A, A)$ as
$\gS^n := S^n \circ \bar{\Phi}^n \circ \iota^n$.
\end{Definition}
Because the cochain complex versions of the Radford map, the Drinfel'd
map and the antipode are cochain isomorphisms by construction, the maps $\gS^n$
are the components of a cochain automorphism of the Hochschild cochain
complex. Its basic property is the following:
\begin{Lemma} \label{LgS} The diagram
\begin{center}
\begin{tikzcd} {}
C^n(A,\prescript{}{\varepsilon}{A}_{\ad}) \arrow{d}{\gS_*}
\arrow{r}{\Omega^n} &
C^n(A, A) \arrow{d}{\gS^n} \\
C^n(A,\prescript{}{\varepsilon}{A}_{\ad}) \arrow{r}{\Omega^n} &
C^n(A, A)
\end{tikzcd}
\end{center}
commutes.
\end{Lemma}
\begin{proof}
This is immediate from Lemma~\ref{iota2}, Definition~\ref{GenDrinf} and
Definition~\ref{GenAntip}: We have
\begin{align*}
\gS^n \circ \Omega^n &= S^n \circ \bar{\Phi}^n \circ \iota^n \circ
\Omega^n =
S^n \circ \bar{\Phi}^n \circ \Omega''^n \circ \iota_* \\
&= S^n \circ \Omega'^n \circ \bar{\Phi}_* \circ \iota_* =
\Omega^n \circ S_* \circ \bar{\Phi}_* \circ \iota_* =
\Omega^n \circ (S \circ \bar{\Phi} \circ \iota)_* =
\Omega^n \circ \gS_*
\end{align*}
since successive composition with~$\iota$, $\bar{\Phi}$ and~$S$ is the
same as composition with~$\gS$. It may be noted that~$\gS$, as the
composition of the bimodule isomorphisms~$\iota$, $\bar{\Phi}$ and~$S$,
is a bimodule automorphism of~$\prescript{}{\varepsilon}{A}_{\ad}$, so
that the map~$\gS_*$ on the left is indeed the component of a cochain map.
\end{proof}

It is not difficult to compute~$\gS^n$ explicitly in terms of the
monodromy matrix~$Q$:
\begin{Cor} \label{CgS}
For $f \in C^n(A, A) = \Hom_K(A^{\ot n}, A)$ and $a_1, \dots, a_n \in
A$, we have
\begin{align*}
&\gS^n(f)(a_1 \ot \dots \ot a_n) =
\rho \left(S(a_{n(2)}) \dots S(a_{1(2)}) f(a_{1(3)} \ot \dots \ot a_{n(3)})
Q_1 \right) \mspace{1mu}
a_{1(1)} \dots a_{n(1)} S(Q_2).
\end{align*}
\end{Cor}
\begin{proof}
We have seen in Section~\ref{Sec:ActCent} that $\gS(a) = \rho(aQ_1)
S(Q_2)$, so that we get
\begin{align*}
\gS^n(f)(a_1 \ot \dots \ot a_n)
&= (\Omega^n \circ \gS_* \circ (\Omega^n)^{-1})(f)(a_1 \ot \dots \ot a_n) \\
&= a_{1(1)} \dots a_{n(1)} \gS((\Omega^n)^{-1}(f)(a_{1(2)} \ot \dots \ot
a_{n(2)})) \\
&= a_{1(1)} \dots a_{n(1)} \gS(S(a_{n(2)}) \dots S(a_{1(2)})
f(a_{1(3)}\ot\dots\ot a_{n(3)})) \\
&= \rho \left(S(a_{n(2)}) \dots S(a_{1(2)}) f(a_{1(3)} \ot \dots \ot
a_{n(3)}) Q_1\right) \mspace{1mu}
a_{1(1)} \dots a_{n(1)} S(Q_2)
\end{align*}
from the preceding lemma.
\end{proof}

For the second generator~$\t$ of the modular group, we also proceed as
in Section~\ref{Sec:ActCent} and let it act on the cochain groups by
multiplication with the ribbon element~$v$:
\begin{Definition}
We define $\T^n \colon C^n(A, A) \to C^n(A, A)$ as
\[\T^n(f)(a_1 \ot \dots \ot a_n) := v f(a_1 \ot \dots \ot a_n)\]
for $f \in C^n(A, A) = \Hom_K(A^{\ot n}, A)$. In other words, using the
map~$\T$ introduced in Section~\ref{Sec:ActCent}, we set $\T^n:=\T_*$,
the composition with~$\T$ on the left.
\end{Definition}

Because~$v$ is central, the maps~$\T^n$ commute with the
differentials~$d^n$ and therefore constitute the components of a cochain
map. The centrality of~$v$ also implies that the diagram
\begin{center}
\begin{tikzcd} {}
C^n(A,\prescript{}{\varepsilon}{A}_{\ad}) \arrow{d}{\T_*}
\arrow{r}{\Omega^n} &
C^n(A, A) \arrow{d}{\T^n} \\
C^n(A,\prescript{}{\varepsilon}{A}_{\ad}) \arrow{r}{\Omega^n} &
C^n(A, A)
\end{tikzcd}
\end{center}
is commutative. Here, it is understood that~$\T_*$ is also given
on~$C^n(A,\prescript{}{\varepsilon}A_{\ad})$ by multiplication with~$v$,
and not via the left or right action of~$v$
on~$\prescript{}{\varepsilon}A_{\ad}$.

The key result that relates these maps to the modular group is the
following theorem:
\begin{Theorem} \label{sts}~
\begin{thmlist}
\item
We have $\gS^n \circ \T^n \circ \gS^n = \rho(v) \; (\T^n)^{-1} \circ \gS^n
\circ (\T^n)^{-1}$.

\item
The cochain maps with components $(\gS^n)^4$ and $((\rho \ot \rho)(Q))^2 \, \id_{C^n(A,A)}$ are homotopic.
\end{thmlist}
\end{Theorem}
\smallskip
\begin{proof}
As recalled in Proposition~\ref{ModRel}, we have
$\gS \circ \T \circ \gS = \rho(v) \; \T^{-1} \circ \gS \circ \T^{-1}$.
Combining the commutativity of the preceding diagram with
Lemma~\ref{LgS}, we therefore get
\begin{align*}
\gS^n \circ \T^n \circ \gS^n \circ \Omega^n &=
\Omega^n \circ \gS_* \circ \T_* \circ \gS_* =
\Omega^n \circ (\gS \circ \T \circ \gS)_* \\
&= \rho(v) \, \Omega^n \circ (\T^{-1} \circ \gS \circ \T^{-1})_*
= \rho(v) \, \Omega^n \circ \T^{-1}_* \circ \gS_* \circ \T^{-1}_*\\
&=\rho(v) \, (\T^n)^{-1}\circ\gS^n\circ(\T^n)^{-1} \circ \Omega^n.
\end{align*}
Because~$\Omega^n$ is bijective, this proves our first assertion.

To prove the second assertion, it suffices to show that the cochain maps
with components $(\Omega^n)^{-1} \circ (\gS^n)^4 \circ \Omega^n$ and 
$((\rho \ot \rho)(Q))^2 \, \id_{C^n(A,\prescript{}{\varepsilon}A_{\ad})}$ are
homotopic, because~$(\Omega^n)$ is an isomorphism of cochain complexes.
By Lemma~\ref{LgS}, we have $(\Omega^n)^{-1} \circ (\gS^n)^4 \circ \Omega^n
= (\gS^4)_*$, and we also have $\gS^2 = (\rho \ot \rho)(Q) \, \uS^{-1}$ by
Proposition~\ref{ModRel}. Therefore our second assertion will hold if we
can show that the cochain map $(\uS^{-2})_*$ is homotopic to the
identity on the cochain complex~$C(A,\prescript{}{\varepsilon}A_{\ad})$,
or equivalently that the cochain map $(\uS^2)_*$ is homotopic to the identity.

We know from Lemma~\ref{SqRib} that~$\uS^2(a) = \ad(a \ot v)$ for all~$a
\in A$. Because by definition the right action on the bimodule $M :=
\prescript{}{\varepsilon}A_{\ad}$ is given by the right adjoint action,
this means in the notation of Proposition~\ref{LeftRight} that
$(\uS^2)_* = (r_v^M)_*$. Now Proposition~\ref{LeftRight} states
that~$(r_v^M)_*$ is homotopic to~$(l_v^M)_*$. But~$(l_v^M)_*$ is the
identity: For $f \in C^n(A,M) =
\Hom_K(A^{\ot n},\prescript{}{\varepsilon}A_{\ad})$ and $a_1, \dots, a_n
\in A$, we have
\begin{align*}
(l_v^M)_*(f)(a_1 \ot \dots \ot a_n) &= v.f(a_1\ot\dots\ot a_n)
= \varepsilon(v) f(a_1 \ot \dots \ot a_n) = f(a_1 \ot \dots \ot a_n)
\end{align*}
because $\varepsilon(v) = 1$.
\end{proof}

As a consequence, we can generalize the projective action of the modular
group on the center~$Z(A)$ obtained in Corollary~\ref{ProjRepCent},
which is by Example~\ref{Zero} equal
to~$HH^0(A, A)$, to an arbitrary Hochschild cohomology group~$HH^n(A,
A)$. For this, we denote the automorphisms of~$HH^n(A, A)$ induced by
the cochain maps~$(\gS^n)$ and~$(\T^n)$ by~$\overline{\gS^n}$
and~$\overline{\T^n}$, respectively, and by~$P(\overline{\gS^n})$
and~$P(\overline{\T^n})$ we denote the corresponding automorphisms of
the projective
space~$P(HH^n(A, A))$. We then have the following generalization of
Corollary~\ref{ProjRepCent}:

\begin{Cor} \label{ActHochCohom}
There is a unique homomorphism from $\SL(2,\Z)$ to $\PGL(HH^n(A, A))$
that maps~$\s$ to~$P(\overline{\gS^n})$ and~$\t$ to~$P(\overline{\T^n})$.
\end{Cor}

Exactly as in the analogous discussion at the end of
Section~\ref{Sec:ActCent}, this representation of the modular group is
linear, and not only projective, if the unique ribbon-normalized right
integral~$\rho$ satisfies~$\rho(v^{-1}) = \pm 1$, or equivalently $(\rho
\ot \rho)(Q) = \pm 1$.

\end{document}